\documentclass[12pt]{article}
\usepackage{tikz}
\usetikzlibrary{calc,through,backgrounds}
\usepackage{graphicx,enumerate}
\usepackage{subfig} 
\usepackage[justification = centering, labelsep =period]{caption} 

\usepackage{multicol}

\usepackage{bm,amssymb,amsthm,amsfonts,amsmath,latexsym,cite,color, psfrag,graphicx,ifpdf,pgfkeys,pgfopts,xcolor,extarrows}

\newtheorem{theo}{Theorem}[section]

\newtheorem{lem} [theo]{Lemma}
\newtheorem{coro}[theo]{Corollary}
\newtheorem{prop}[theo]{Proposition}
\newtheorem{prob}[theo]{Problem}
\newtheorem{re}[theo]{Remark}
\newtheorem{conj}[theo]{Conjecture}

\allowdisplaybreaks

\makeatletter \@addtoreset{equation}{section}
\@addtoreset{theo}{}\makeatother

\usepackage{hyperref}

\renewcommand{\arraystretch}{1.3}
\newcommand{\BPD}[2][1pc]{%
\setlength{\unitlength}{#1}
\def\BPDframe{%
    \thinlines%
    \color{lightgray}%
    \put(0,0){\line(0,1){1}}%
    \put(1,0){\line(0,1){1}}%
    \put(0,0){\line(1,0){1}}%
    \put(0,1){\line(1,0){1}}%
    \linethickness{0.08\unitlength}%
    \color{teal}}
\def\O{
\begin{picture}(1,1)
    \BPDframe
\end{picture}}
\def\F{
\begin{picture}(1,1)
    \BPDframe
    \qbezier(0.5,0)(0.5,0.2)(0.5,0.2)
    \qbezier(1,0.5)(0.8,0.5)(0.8,0.5)
    \qbezier(0.8,0.5)(0.5,0.5)(0.5,0.2)
\end{picture}}
\def\J{
\begin{picture}(1,1)
    \BPDframe
    \qbezier(0.5,1)(0.5,0.8)(0.5,0.8)
    \qbezier(0,0.5)(0.2,0.5)(0.2,0.5)
    \qbezier(0.5,0.8)(0.5,0.5)(0.2,0.5)
\end{picture}}
\def\Z{
\begin{picture}(1,1)
    \BPDframe
    \qbezier(0.5,0)(0.5,0.2)(0.5,0.2)
    \qbezier(0,0.5)(0.2,0.5)(0.2,0.5)
    \qbezier(0.2,0.5)(0.5,0.5)(0.5,0.2)
\end{picture}}
\def\L{
\begin{picture}(1,1)
    \BPDframe
    \qbezier(0.5,1)(0.5,0.8)(0.5,0.8)
    \qbezier(1,0.5)(0.8,0.5)(0.8,0.5)
    \qbezier(0.5,0.8)(0.5,0.5)(0.8,0.5)
\end{picture}}
\def\I{
\begin{picture}(1,1)
    \BPDframe
    \qbezier(0.5,0)(0.5,0.5)(0.5,1)
\end{picture}}
\def\H{
\begin{picture}(1,1)
    \BPDframe
    \qbezier(0,0.5)(0.5,0.5)(1,0.5)
\end{picture}}
\def\x{
\begin{picture}(1,1)
    \BPDframe
    \qbezier(0,0.5)(0.5,0.5)(1,0.5)
    \qbezier(0.5,0)(0.5,0.3)(0.5,0.3)
    \qbezier(0.5,1)(0.5,0.7)(0.5,0.7)
\end{picture}}
\def\X{
\begin{picture}(1,1)
    \BPDframe
    \qbezier(0.5,0)(0.5,0.5)(0.5,1)
    \qbezier(0,0.5)(0.5,0.5)(0.5,0.5)
    \qbezier(1,0.5)(0.5,0.5)(0.5,0.5)
\end{picture}}
\def\B{
\begin{picture}(1,1)
    \BPDframe
    \qbezier(0.5,0)(0.5,0.2)(0.5,0.2)
    \qbezier(1,0.5)(0.8,0.5)(0.8,0.5)
    \qbezier(0.8,0.5)(0.5,0.5)(0.5,0.2)
    \qbezier(0.5,1)(0.5,0.8)(0.5,0.8)
    \qbezier(0,0.5)(0.2,0.5)(0.2,0.5)
    \qbezier(0.5,0.8)(0.5,0.5)(0.2,0.5)
\end{picture}}
\def\b{
\begin{picture}(1,1)
    \BPDframe
    \qbezier(0.5,0)(0.5,0.2)(0.5,0.2)
    \qbezier(0,0.5)(0.2,0.5)(0.2,0.5)
    \qbezier(0.2,0.5)(0.5,0.5)(0.5,0.2)
    \qbezier(0.5,1)(0.5,0.8)(0.5,0.8)
    \qbezier(1,0.5)(0.8,0.5)(0.8,0.5)
    \qbezier(0.5,0.8)(0.5,0.5)(0.8,0.5) 
\end{picture}}
\def\C{
\begin{picture}(1,1)
    \BPDframe
    \qbezier(0,0)(0.5,0.5)(1,1)
    \qbezier(1,0)(0.5,0.5)(0,1)
\end{picture}}
\def\M##1{
\BPDframe
\begin{picture}(1,1)%
    \put(0,0.2){\makebox[\unitlength]{\(##1\)}}
\end{picture}}
\def\N##1{
\begin{picture}(1,1)%
    \put(0,0.2){\makebox[\unitlength]{\(##1\)}}
\end{picture}}
\def\t##1{\begin{picture}(1,1)%
    \BPDframe
    \put(0,0.28){\makebox[\unitlength]{\(##1\)}}
    \qbezier(0.5,0.5)(1,0.5)(1,0.5)
    \qbezier(0.5,0.5)(0.5,0)(0.5,0)
\end{picture}}
\def\tt##1{\begin{picture}(1,1)%
    \BPDframe
    \put(0,0.35){\makebox[\unitlength]{\(##1\)}}
    \qbezier(0.5,0.5)(1,0.5)(1,0.5)
    \qbezier(0.5,0.5)(0.5,0)(0.5,0)
\end{picture}}
\def\G{\begin{picture}(1,1)%
    \BPDframe
    \put(0.0,0.18){\colorbox{gray}{\begin{picture}(.65,.65)
    \end{picture}}}
\end{picture}}
\def\g{\begin{picture}(1,1)%
    \BPDframe
    \put(0.0,0.25){\colorbox{gray}{\begin{picture}(.5,.5)
    \end{picture}}}
\end{picture}}
\def\gg{\begin{picture}(1,1)%
    \BPDframe
    \put(0.0,0.22){\colorbox{gray}{\begin{picture}(.57,.57)
    \end{picture}}}
\end{picture}}
\def\R{\begin{picture}(1,1)%
    \BPDframe
    \put(0.0,0.22){\colorbox{green}{\begin{picture}(.57,.57)
    \end{picture}}}
\end{picture}}
\def\r{\begin{picture}(1,1)%
    \BPDframe
    \put(0.0,0.25){\colorbox{green}{\begin{picture}(.5,.5)
    \end{picture}}}
\end{picture}}
\def\IC{
\begin{picture}(1,1)
    \BPDframe
    \put(0.0,0.05){{\footnotesize \color{black} $\checkmark$}}
    \qbezier(0.5,0)(0.5,0.5)(0.5,1)
\end{picture}}
\def\ggC{\begin{picture}(1,1)%
    \BPDframe
    \put(0.0,0.22){\colorbox{gray}{\begin{picture}(.57,.57)
    \end{picture}}}
    \put(0.0,0.05){{\footnotesize \color{black} $\checkmark$}}
\end{picture}}
\def\XC{
\begin{picture}(1,1)
    \BPDframe
    \qbezier(0.5,0)(0.5,0.5)(0.5,1)
    \qbezier(0,0.5)(0.5,0.5)(0.5,0.5)
    \qbezier(1,0.5)(0.5,0.5)(0.5,0.5)
    \put(0.0,0.05){{\footnotesize \color{black} $\checkmark$}}
\end{picture}}
\def\HC{
\begin{picture}(1,1)
    \BPDframe
    \qbezier(0,0.5)(0.5,0.5)(1,0.5)
    \put(0.0,0.05){{\footnotesize \color{black} $\checkmark$}}
\end{picture}}
\def \T{\t{\bullet}}
\def \TT{\tt{\bullet}}
\def \TC{\put(0.0,0.05){{\footnotesize \color{black} $\checkmark$}}\t{\bullet}}
\begin{array}{@{\,}c@{\,}}{}
{\def\arraystretch{0}
\setlength{\arraycolsep}{0pc}
\color{teal}
\begin{array}{@{}l@{}}%
#2\end{array}}
\end{array}}

\def\S{  \mathfrak{S}}
\def\N{  \mathrm{N}}

\def\r2{r_2(i_b,i_c,j,j')}

\begin{document}
\begin{center}
{\Large\bf Schubert polynomials and patterns in permutations}

\vskip 6mm
{\small   }
Peter L. Guo  and Zhuowei Lin

\end{center}

\begin{abstract}
This paper investigates the number of supports of the Schubert polynomial $\mathfrak{S}_w(x)$ indexed by a permutation $w$. This number also equals the number of lattice points in the Newton polytope of $\mathfrak{S}_w(x)$. We  establish  a lower bound for this number in terms of the occurrences of patterns in $w$. The analysis is carried out in the general framework of   dual characters of  flagged Weyl modules. 
Our result considerably  improves  the bounds for principal specializations of Schubert polynomials or dual flagged Weyl characters  previously obtained by Weigandt, Gao, and  M{\'e}sz{\'a}ros--St. Dizier--Tanjaya.
Some problems and conjectures     are     discussed.

\end{abstract}

\vskip 3mm


\vskip 3mm


\newcommand*\cir[1]{\tikz[baseline=(char.base)]{
            \node[shape=circle,draw,inner sep=2pt] (char) {#1};}}

\section{Introduction}

As usual, let $S_n$ be the symmetric group of permutations of $[n]:=\{1,2,\ldots, n\}$. Gvien a permutation   $w\in S_n$, let $\mathfrak{S}_w(x)$ denote  the associated Schubert polynomial. They  were introduced by 
Lascoux  and Sch\"utzenberger \cite{schubert} to represent  Schubert classes in the cohomology ring of the flag manifold. Schubert polynomials   can be defined in a recursive procedure. For the longest permutation $w_0=n\cdots 2 1$, set $\mathfrak{S}_{w_0}(x)=x_1^{n-1}x_2^{n-2}\cdots x_{n-1}$. For $w\neq w_0$, locate a position  $1\leq i<n$ such that $w(i)<w(i+1)$, and   set $\mathfrak{S}_w(x)=\partial_i \mathfrak{S}_{ws_i}(x)$, where $ws_i$ is  obtained from $w$ by swapping $w(i)$ and $w(i+1)$, and $\partial_i$ is the divided difference operator acting on a polynomial $f(x)$ by 
\[
\partial_i f(x)=\frac{f(x)-f(x)|_{x_i\leftrightarrow x_{i+1}}}{x_i-x_{i+1}}.
\]

For a (weak) composition 
$\alpha=(\alpha_1,\ldots, \alpha_n)\in \mathbb{Z}_{\geq 0}^n$, write  $x^\alpha=x_1^{\alpha_1}\cdots x_n^{\alpha_n}$. Then one can express
\[
\mathfrak{S}_w(x)=\sum_{\alpha\in \mathbb{Z}_{\geq 0}^n} a_\alpha\, x^\alpha.
\]
It is famously known that $a_\alpha\in \mathbb{Z}_{\geq 0}$ \cite{Mac}. 
We say that  $\alpha$ is a support of $\mathfrak{S}_w(x)$ if $a_\alpha>0$. By the work of  Fink,  M{\'e}sz{\'a}ros and   St. Dizier \cite{2018Schubert} (first conjectured by Monical,  Tokcan and   Yong \cite{MTY}),   the supports of $\mathfrak{S}_w(x)$ are in one-to-one correspondence with  lattice points in its Newton polytope. Recall that the Newton polytope of a polynomial $f$ in $x_1,\ldots, x_n$ is the convex hull in $\mathbb{R}^n$ generated by the supports of $f$. 

 We use  $\theta_w$ to stand for the number of supports of $\mathfrak{S}_w(x)$, or equivalently, the number of lattice points in the Newton polytope of $\mathfrak{S}_w(x)$.  Given $u=u(1)\cdots u(m)\in S_m$ with $m\leq n$, we say that a subsequence $w(i_1)\cdots w(i_m)$    of $w$ is  a $u$ pattern if  $w(i_1)\cdots w(i_m)$ has the same relative order as $u$.
Let $p_u(w)$ denote the number of appearances of  $u$ patterns  in $w$. For example, we have $p_{132}(1432)=3$. 

Our first main result is an attempt to  establish a lower bound for $\theta_w$ in terms of  the numbers $p_u(w)$.

\begin{theo}\label{schres}
For $w\in S_n$, we have 
\begin{align}\label{mkod-1}
        \theta_w&\ge  1+p_{132}(w)+p_{1432}(w)+p_{13254}(w)+3p_{14253}(w)
        \nonumber\\
        &\ \ \ \ +p_{14352}(w)+4
p_{15243}(w)+p_{15324}(w)+2p_{15342}(w)
        \\
        &\ \ \ \ +p_{15432}(w)+p_{24153}(w)+2p_{25143}(w)+p_{35142}(w).\nonumber
    \end{align}
        
\end{theo}

As comparison,   the principal specialization $\nu_w:=\mathfrak{S}_w(x)|_{x_i=1}$ of $\mathfrak{S}_w(x)$ has received much attention in recent years. By definition, it is clear that $\theta_w\leq \nu_w$. The equality holds if and only if $\mathfrak{S}_w(x)$ is  zero-one, that is,  each coefficient $a_\alpha$ is equal to either 0 or 1. A criterion for zero-one Schubert polynomials was first given  by  Fink,  M{\'e}sz{\'a}ros and   St. Dizier \cite{fink2021zero}.

A classical formula due to Macdonald \cite{Mac}, see also \cite{FS,HPSW},  states   that 
\[
\nu_w=\frac{1}{\ell !}\sum_{(a_1,\ldots, a_\ell)\in \mathrm{Red}(w)} a_1\cdots a_\ell,
\]
where $\ell$ is the length of $w$, and  the sum runs over reduced words of $w$. It is well known that $\nu_w=1$ if and only if $w$ is dominant, that is, $w$ has no 132 pattern.   Stanley \cite{stanley} conjectured that $\nu_w=2$ if and only if $w$ has exactly one  132 pattern. This conjecture was confirmed  by  Weigandt \cite{Anna} by proving  a lower bound 
\begin{equation}\label{cbuAW}
\nu_w\geq 1+p_{132}(w).
\end{equation} 
This bound was later  strengthened by Gao  \cite[Theorem 2.1]{gao}, where it was shown that 
\begin{equation}\label{GAO-b}
    \nu_w\geq 1+p_{132}(w)+p_{1432}(w).
\end{equation}
Both proofs in \cite{Anna,gao} make use of the pipe dream model of Schubert polynomials. 

Because of  $\theta_w\leq \nu_w$, Theorem \ref{schres} immediately yields a lower bound for $\nu_w$ which largely improves the bound in  \eqref{GAO-b}.  

\begin{coro}
For $w\in S_n$, $\nu_w$ is  bounded below by the right-hand side of    \eqref{mkod-1}.
\end{coro}

We deal with Theorem \ref{schres} in the general setting of dual characters of flagged Weyl modules associated to diagrams in the square grid $[n]\times [n]$.  
A diagram $D$ means a subset of boxes in $[n]\times [n]$.  The associated flagged Weyl module $\mathcal{M}_{D}$ is a representation of the Borel group $B$ of invertible upper-triangular complex  matrices \cite{KP-1,KP-2,magyar1998schubert}. 
Let $\chi_D(x)=\chi_D(x_1,\ldots, x_n)$ denote the dual character of  $\mathcal{M}_{D}$. As will be explained in Section \ref{section-21}, $\chi_D(x)$ specializes  to a Schubert polynomial (resp., key polynomial)  when $D$ is the Rothe diagram  of a permutation (resp., skyline diagram of a composition). 

We use $\theta_D$ to represent the number  of supports of $\chi_D(x)$, which also equals the number of lattice points in the Newton polytope, called Schubitope,  of $\chi_D(x)$ \cite{2018Schubert}. 
We deduce   a lower bound for $\theta_D$ by using the appearances of certain subdiagrams of $D$. Precisely, consider the configurations in Figure \ref{config}, where a blank/shaded box means the absence/presence. 
\vspace{0.05cm}
\begin{figure}[h]
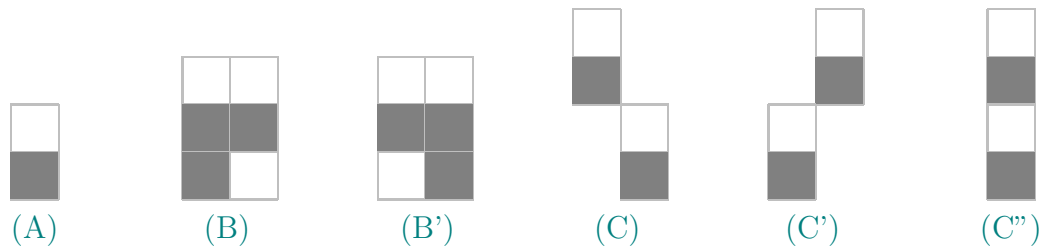

\vspace{-0.4cm}
  \centering
  ~~~~~~~~
    \begin{minipage}[b]{0.1\textwidth}
    $$\BPD[1.5pc]{\O\\
    \G\\
    \\
    \N{\text{(A)}}\\}
    $$
    \label{figA}
  \end{minipage}
  \hfill
  \begin{minipage}[b]{0.1\textwidth}
    $$\BPD[1.5pc]{
    \O\O\\
    \G\G\\
    \G\O\\
    \\
    \N{~~~~\text{(B)}}\\}
    $$
    \label{figA}
  \end{minipage}
  \hfill
      \begin{minipage}[b]{0.1\textwidth}
    $$\BPD[1.5pc]{\O\O\\
    \G\G\\
    \O\G\\
    \N{~~~~~\text{(B')}}\\}
    $$
    \label{figB'}
  \end{minipage}
  \hfill
  \begin{minipage}[b]{0.1\textwidth}
    $$\BPD[1.5pc]{
    \O\N{}\\
    \G\N{}\\
    \N{}\O\\
    \N{}\G\\
    \N{~~~~\text{(C)}}\\}$$
    \label{figB}
  \end{minipage}
  \hfill
    \begin{minipage}[b]{0.1\textwidth}
    $$\BPD[1.5pc]{
    \N{}\O\\
    \N{}\G\\
    \O\N{}\\
    \G\N{}\\
    \N{~~~~\text{(C')}}\\}$$
    \label{figB}
  \end{minipage}
  \hfill
  \begin{minipage}[b]{0.1\textwidth}
    $$\BPD[1.5pc]{
    \O\\
    \G\\
    \O\\
    \G\\
\N{\text{(C'')}}\\}$$
      \label{fig C}
  \end{minipage}
\vspace{-10pt}
  \caption{Configurations for Theorem \ref{res1}.}
    \label{config}
\end{figure}
 Let $(i,j)$   denote the box of $[n]\times [n]$ in row $i$ and column $j$ in matrix coordinate.  Define
\begin{itemize}
\item $r_1(D)$: the number of subdiagrams of $D$ which are equal to the configuration (A) in Figure \ref{config};

\item $r_2(D)$: the number of subdiagrams of $D$ which are equal to the configuration (B), or (B') in Figure \ref{config};

\item $r_3(D)$: the number of subdiagrams of $D$ which are equal to the configuration (C), or (C'), or (C'') in Figure \ref{config};

\end{itemize}

To avoid confusion, we explain the above notions  in more details.  For example, a subdiagram of $D$ which is equal to the configuration  (C) in Figure \ref{config} means a subset 
\[
\{(i_1, j_1), (i_2, j_1), (i_3, j_2), (i_4, j_2)\}
\]
of boxes in $[n]\times [n]$ such that (1) $i_1<i_2<i_3<i_4$ and $j_1<j_2$, and (2) exactly two of the boxes, $(i_2, j_1)$ and $(i_4, j_2)$,  belong  to $D$.    

The statistic $r_1(D)$ has be investigated   by M{\'e}sz{\'a}ros,  St. Dizier and  Tanjaya \cite{meszaros2021principal}, where it is called the rank of $D$ and is denoted $\mathrm{rank}(D)$.   Let $\chi_D(1,\ldots,1) :=\chi_D(x)|_{x_i=1}$ be the principal specialization of $\chi_D(x)$. Note that    $\chi_D(1,\ldots,1)\geq \theta_D$. A criterion for the equality was conjectured in  \cite{meszaros2021principal}  and proved recently  in  \cite{GLP}. As  shown in \cite[Theorem 2]{meszaros2021principal},  $\chi_D(1,\ldots,1)$ is bounded below by $1+r_1(D)$, which recovers  the bound in \eqref{cbuAW} by Weigandt \cite{Anna}  when $D$ is the Rothe diagram of a permutation.

We prove the following lower bound for $\theta_D$.

\begin{theo}\label{res1}
For any diagram $D$, we have
\begin{equation}\label{bnkat-1}
\theta_D\ge 1+r_1(D)+r_2(D)+r_3(D).   
\end{equation}
\end{theo}

Theorem \ref{res1}    leads to a  strengthen of the above mentioned  bound for $\chi_D(1,\ldots,1)$ by M{\'e}sz{\'a}ros,  St. Dizier and  Tanjaya \cite[Theorem 2]{meszaros2021principal}.

\begin{coro}\label{res1c}
    For any diagram $D$, $\chi_D(1,\ldots,1)$ is bounded below by the right-hand side of \eqref{bnkat-1}.
\end{coro}

When   restricting to the skyline diagram $D(\alpha)$  of a composition $\alpha\in \mathbb{Z}_{\geq 0}^n$,  Theorem \ref{res1}  yields  a lower bound for the number of supports of the key polynomial $\kappa_\alpha(x)$.  Key polynomials,  also called the Demazure characters, are the characters of a Demazure modules
 for the general linear group  \cite{D-1,D-2}. 
Key polynomials can also be defined in a recursive manner.  
If $\alpha=(\alpha_1\geq \cdots  \geq \alpha_n)$ is weakly decreasing, then set $\kappa_\alpha(x)=x^\alpha$. Otherwise, choose $1\leq i<n$ such
 that $\alpha_i<\alpha_{i+1}$, and set $\kappa_\alpha(x)=\partial_i\,x_i\,\kappa_{s_i\alpha}(x)$, where $s_i\alpha$ is obtained  from $\alpha$ by swapping 
 $\alpha_i$ and $\alpha_{i+1}$.

\begin{theo}\label{forkey}
For any (weak) composition $\alpha$,   we have
\begin{align*}
\kappa_\alpha(1,\ldots, 1) \ge \theta_{D(\alpha)}\ge & 1+ \sum_{\mathrm{inv}_1(\alpha)}(\alpha_{i_2}-\alpha_{i_1})+\sum_{\mathrm{inv}_2(\alpha)}(\alpha_{i_2}-\alpha_{i_3})\cdot(\alpha_{i_3}-\alpha_{i_1})\\[5pt]
   &
\ +\sum_{\mathrm{inv}_{3}(\alpha)}(\alpha_{i_2}-\alpha_{i_1})\cdot(\alpha_{i_4}-\alpha_{i_3}),
    \end{align*}
where $\mathrm{inv}_1(\alpha)=\{(i_1,i_2)\colon i_1<i_2,\ \alpha_{i_1}<\alpha_{i_2}\}$, $\mathrm{inv}_2(\alpha)=\{(i_1, i_2, i_3)\colon i_1<i_2<i_3, \ \alpha_{i_1}<\alpha_{i_3}<\alpha_{i_2}\}$, and $\mathrm{inv}_{3}(\alpha)=\{(i_1, i_2, i_3, i_4)\colon i_1<i_2<i_3<i_4,\ \alpha_{i_1}<\alpha_{i_2},\ \alpha_{i_3}<\alpha_{i_4}\}$.
\end{theo}

Taking only the first  summation, Theorem  \ref{forkey} reduces to the following lower bound by 
M{\'e}sz{\'a}ros,  St. Dizier and  Tanjaya \cite[Corollary 20]{meszaros2021principal}:
\begin{equation*} 
\kappa_\alpha(1,\ldots, 1)\geq  1+ \sum_{\mathrm{inv}_1(\alpha)}(\alpha_{i_2}-\alpha_{i_1}).
\end{equation*} 

This paper is organized as follows. Section \ref{section-21} lays out basic information that we need about  the dual characters of flagged Weyl modules.  Section \ref{Finj} is devoted to a proof of  Theorem \ref{res1}, based on which we complete the proofs of Theorems \ref{schres} and \ref{forkey} in Section \ref{paoe}. We conclude in Section \ref{condlu-o9} with some  problems and  conjectures. We put some tables that are needed in the proof of Theorem \ref{schres} in the  appendix section. 

\subsection*{Acknowledgements}
We would like to thank Yibo Gao for helpful comments. 
This work was  supported by the National Natural Science Foundation of China (No. 12371329) and the Fundamental Research Funds for the Central Universities (No. 63243072).

\section{Dual characters of flagged Weyl modules}\label{section-21}

In this section, we review some necessary background on flagged Weyl modules, and explain   how their  dual characters specialize to Schubert and key polynomials.

Recall that a diagram $D$ is a subset of boxes in $[n]\times [n]$. Write $D=(D_{1}, D_{2},\ldots, D_{n})$, where, for $1\leq j\leq n$, $D_j$ denotes  the $j$-th column of $D$. 
We  also represent $D_j$ by  a subset of $[n]$, that is,   $i\in D_j$
 if and only if the box $(i,j)$ belongs to $D$. 
 For example, the diagram in Figure \ref{Buys}  can be expressed   as $D=(\{2,3,4\},\emptyset,\{1,2\},\{3\})$.
\begin{figure}[h]
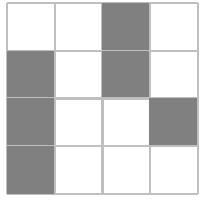

    \centering
$$\BPD[1.5pc]{
\O\O\G\O\\
\G\O\G\O\\
\G\O\O\G\\
\G\O\O\O\\

}$$
\vspace{-5pt}
\caption{A diagram in  $[4]\times [4]$.}
\label{Buys}   
\end{figure}

Let $R=\{r_1< \cdots<r_k\}$ and $S= \{s_1< \cdots<s_k\} $ be  two $k$-element  subsets  of $[n]$. We say $R\leq S$  if  $r_i\leq s_i$ for $1\leq i\leq k$. This defines a  partial order on all $k$-element subsets of $[n]$, which  is usually  called the {  Gale order}.  
For two diagrams $C=(C_{1}, \ldots, C_{n})$ and $D=(D_{1}, \ldots,D_{n})$,  denote   $C\leq D$ if $C_{j}\leq D_{j}$ for each $1\leq j\leq n$. 

Let $\mathrm{GL}(n,\mathbb{C})$ be the general linear group of $n\times n$ invertible complex matrices, and $B$  the Borel subgroup of $\mathrm{GL}(n,\mathbb{C})$ which consists of   upper-triangular matrices. We use  $Y$ to denote  the upper-triangular matrix of variables $y_{ij}$ with $1\leq i\leq j\leq n$:
\begin{equation*}
	Y=
	\begin{bmatrix}
		y_{11}&y_{12}&\ldots&y_{1n}\\
		0&y_{22}&\ldots&y_{2n}\\
		\vdots&\vdots&\ddots&\vdots\\
		0&0&\ldots&y_{nn}
	\end{bmatrix}.
\end{equation*}
Let   $\mathbb{C} [Y]$   be the linear space of polynomials  in   $\{y_{ij}\}_{i\leq j}$ over $\mathbb{C}$. Define the (right) action of $B$   on $\mathbb{C} [Y]$ by   $f(Y)\cdot b=f(b^{-1}\cdot Y)$, where $b\in B$ and  $f\in \mathbb{C} [Y]$. For two subsets $R$ and $S$ of $[n]$ with the same cardinality, let  $Y^{R}_{S}$ be  the submatrix of $Y$ with rows indexed by $R$ and columns indexed by $S$. Note  that $\det\left(Y^{R}_{S}\right)\neq 0$ if and only if $R\leq S$.
For $C=(C_{1}, \ldots, C_{n})$ and  $D=(D_{1}, \ldots, D_{n})$ with $C\leq D$, denote 
\[
\mathrm{det}\left(Y_{D}^C\right)=\prod_{j=1}^{n}\mathrm{det}\left(Y_{D_{j}}^{C_{j}}\right).
\]
The flagged Weyl module  associated to  $D$ is the subspace 
\begin{align*}	\mathcal{M}_{D}=\mathrm{Span}_{\mathbb{C}}\left\{\mathrm{det}\left(Y_D^C\right)\colon C\leq D\right\},
\end{align*}
which  is a $B$-module with the action inherited from the action of $B$ on $\mathbb{C} [Y]$.
 
Let $X=\mathrm{diag}(x_1,\ldots, x_n)$ be the diagonal matrix.
The character of $\mathcal{M}_{D}$ is defined as
\[\mathrm{char}(\mathcal{M}_{D})(x_{1},\ldots, x_{n})=\mathrm{tr}(X\colon \mathcal{M}_{D}\rightarrow \mathcal{M}_{D}).\] 
It is readily  checked  that  $\mathrm{det}\left(Y_{D}^C\right)$ for $C\leq D$  is an
eigenvector of $X$ with eigenvalue
\[\prod_{j=1}^n\prod_{i\in C_j}x_i^{-1}.\]
The    dual character is defined  to be 
$ \chi_{D}(x):=\mathrm{char}(\mathcal{M}_{D})(x_{1}^{-1},\ldots,x_{n}^{-1}).$ 
The weight vector $\mathrm{wt}(C)=(\alpha_1,\ldots, \alpha_n)$ of a diagram $C=(C_1,\ldots, C_n)$ is defined by letting $\alpha_i$ be the number of appearances of $i$ in   $C_1,\ldots, C_n$. 
Geometrically, $\alpha_i$ is the nubmer of boxes lying  in row $i$. 
The diagram in Figure \ref{Buys} has weight vector $(1, 2, 2,1)$. By the above arguments, we have the following characterization on the supports of $\chi_D(x)$, see  Adve,    Robichaux  and   Yong \cite{ARY}  for discussions  about the computational complexity for deciding the supports. 

\begin{prop}\label{gapnj-1}
  The set of supports of $\chi_D(x)$ is $\{\mathrm{wt}(C)\colon C\leq D\}$.   
\end{prop}

The Rothe diagram  $D(w)$ of a permutation $w\in S_n$ can be constructed  as follows. For $1\leq i\leq n$, place  a pot in row $i$ and column $w(i)$. Then $D(w)$ is obtained by ignoring all boxes to the right of a dot in the same row, and all boxes below a dot in the same column. The skyline diagram $D(\alpha)$ of a composition $\alpha=(\alpha_1,\ldots, \alpha_n)$ consists of the leftmost $\alpha_i$ boxes for each $1\leq i\leq n$.   Figure \ref{RSjlog} displays  the Rothe diagram of $w=1432$ and the skyline diagram of $\alpha=(1,3,0,2)$. 
 \begin{figure}[h]
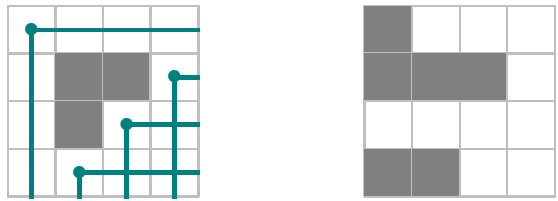

    \centering
    \begin{minipage}[b]{0.3\textwidth}
    $$
    \BPD[1.5pc]{
    \TT\H\H\H\\
    \I\G\G\TT\\
    \I\G\TT\X\\
    \I\TT\X\X\\
    }
    $$
    \end{minipage}
        \begin{minipage}[b]{0.3\textwidth}
    $$
    \BPD[1.5pc]{
    \G\O\O\O\\
    \G\G\G\O\\
    \O\O\O\O\\
    \G\G\O\O\\

    }
    $$
    \end{minipage}
    \caption{The Rothe diagram of $w=1432$ and the skyline diagram of $\alpha=(1,3,0,2)$.}\label{RSjlog}
\end{figure}
 As specializations, it is well known that $\chi_{D(w)}(x)=\mathfrak{S}_w(x)$ \cite{KP-1, KP-2} and $\chi_{D(\alpha)}(x)=\kappa_\alpha(x)$ \cite{D-2}.

\section{Proof of Theorem \ref{res1}}\label{Finj}

Let $D$ be a given diagram in $[n]\times [n]$. 
For  simplicity, we denote $r_i=r_i(D)$ for $i=1,2,3$. 
To finish the proof of Theorem \ref{res1}, by Proposition \ref{gapnj-1}, it suffices to  construct $r_1+r_2+r_3$ diagrams less than $D$ which have distinct  weight vectors.  To accomplish  this, we shall design three algorithms to produce  such diagrams.

\subsection{The first algorithm}\label{Sub-11}

The first algorithm produces a chain $C^1<C^2<\cdots<C^{r_1}<D$, including $r_1$ diagrams less than $D$. It  essentially obeys a similar idea to that in the proof of  \cite[Lemma 17]{meszaros2021principal}.

For $(i,j)\in D$, denote by $r_1(i,j)$ the number of blank boxes  above $(i,j)$ in the same column, namely,
$$
r_1(D;i,j)=\#\left\{i'\colon  i'<i,(i',j)\notin D\right\}.
$$
By definition, it follows  that 
\[
r_1=\sum_{(i,j)\in D}r_1(D;i,j).
\]

{\bf Algorithm 1}. 
First, we construct $C^{r_1}$. Among the boxes $(i, j)$ of $D$ such that $r_1(D; i,j)>0$, choose the top-left most one, say $(i_0, j_0)$.  Note that the box $(i_0-1, j_0)$ right above it must be blank.  Set $C^{r_1}=D\setminus \{(i_0, j_0)\}\cup \{(i_0-1,j_0)\}$. It is clear that $C^{r_1}<D$ and $r_1(C^{r_1})=r_1(D)-1$. 
Replacing  $D$ by $C^{r_1}$, we are given $C^{r_1-1}<C^{r_1}$ with $r_1(C^{r_1-1})=r_1(C^{r_1})-1$. Repeating this procedure yields  the desired chain $C^1<C^2<\cdots<C^{r_1}<D$.  An illustration for this algorithm is depicted in Figure \ref{VGoanj}, where the position  marked with a crossing means the box of $D$ which is moved up.
\begin{figure}[ht]
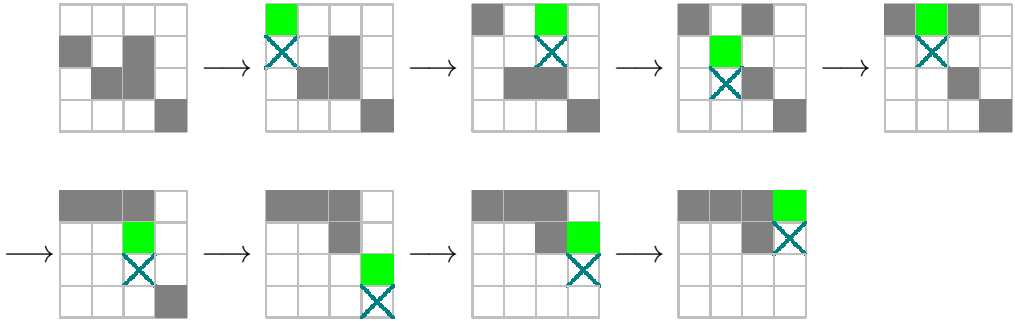

    \centering
        \begin{align*}
&\BPD{
\O\O\O\O\\
\g\O\g\O\\
\O\g\g\O\\
\O\O\O\g\\
}\longrightarrow
\BPD{
\r\O\O\O\\
\C\O\g\O\\
\O\g\g\O\\
\O\O\O\g\\
}\longrightarrow
\BPD{
\g\O\r\O\\
\O\O\C\O\\
\O\g\g\O\\
\O\O\O\g\\
}\longrightarrow
\BPD{
\g\O\g\O\\
\O\r\O\O\\
\O\C\g\O\\
\O\O\O\g\\
}\longrightarrow
\BPD{
\g\r\g\O\\
\O\C\O\O\\
\O\O\g\O\\
\O\O\O\g\\
}\\
~
\\
\longrightarrow
&\BPD{
\g\g\g\O\\
\O\O\r\O\\
\O\O\C\O\\
\O\O\O\g\\
}
\longrightarrow
\BPD{
\g\g\g\O\\
\O\O\g\O\\
\O\O\O\r\\
\O\O\O\C\\
}\longrightarrow
\BPD{
\g\g\g\O\\
\O\O\g\r\\
\O\O\O\C\\
\O\O\O\O\\
}\longrightarrow
\BPD{
\g\g\g\r\\
\O\O\g\C\\
\O\O\O\O\\
\O\O\O\O\\
}
\end{align*}
    \caption{An illustration for performing  Algorithm 1.}
   \label{VGoanj}
\end{figure}

\begin{re}
The  subdiagrams produced by Algorithm 1 are generally   different from the subdiagrams constructed  in \cite[Lemma 17]{meszaros2021principal}.
The reason that we adopt Algorithm 1 is that we need to produce subdiagrams whose weight vectors are distinct from those generated by Algorithm 2 and Algorithm 3 in  the next two  subsections.  
\end{re}

\begin{prop}\label{OBNHG}
The weight vectors of the $r_1$ diagrams  generated  in Algorithm 1 are   distinct.   
\end{prop}

\begin{proof}
As explained  in  \cite[Lemma 18]{meszaros2021principal}, if $C<D$, then $C$ and $D$ have distinct weight vectors.  This allows us to conclude  the proof.
\end{proof}

\subsection{The second algorithm}\label{VHIHI-1}

The second algorithm will give rise to $r_2$ diagrams less than $D$   with distinct weight vectors. Moreover, these weight vectors are distinct from the weight vectors of diagrams  generated  by Algorithm 1. 

Given a pair of row indices  $1< i_1<i_2\leq n$ and a pair of column indices $1\leq j_1<j_2\leq n$,  let $r_2(D;i_1, i_2; j_1, j_2)$ denote the number  of row indices $i$ with $i<i_1$  such that the  subdiagram of $D$, which includes  the six boxes restricted to rows $\{i, i_1, i_2\}$ and columns $\{j_1, j_2\}$, is either  the configuration (B), or (B') in Figure \ref{config}.  
Summing over all possible row and column index pairs, we see that 
$$r_2=\sum_{\substack{1< i_1<i_2\le n\\
1\le j_1<j_2\le n}} r_2(D;i_1,i_2; j_1,j_2).
$$

The second algorithm will be well understood via an example by taking  $D$ to be the diagram in Figure \ref{BUJUF-3}. 
\begin{figure}[ht]
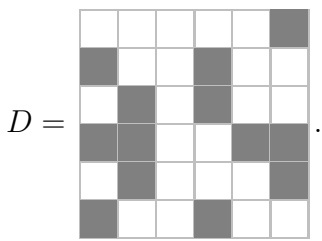


    \centering
    \vspace{-0.3cm}
  $$D=\BPD[1.2pc]{
    \O\O\O\O\O\gg\\
    \gg\O\O\gg\O\O\\
    \O\gg\O\gg\O\O\\
    \gg\gg\O\O\gg\gg\\
    \O\gg\O\O\O\gg\\
    \gg\O\O\gg\O\O\\
    }.$$
\vspace{-0.3cm}
 \caption{A diagram $D$ in $[6]\times [6]$.}
   \label{BUJUF-3}
\end{figure}
   
\textbf{Algorithm 2.} The algorithm will be performed for any given  pair of row indices $1<i_1<i_2\leq n$. Running over all   row index pairs, we receive  the desired $r_2$ diagrams less the $D$ with distinct weights.

Step 0: Let us fix $1<i_1<i_2\leq n$. Set $S_{i_1, i_2}(D)=\emptyset$. 
For each column of $D$, consider the boxes  of $D$ above row $i_1$. Move these boxes up along the column such that they occupy the topmost positions. For example, if we take $(i_1, i_2)=(4,5)$, then the resulting diagram of $D$ after Step 0 is illustrated in Figure \ref{Step-0}. 
\begin{figure}[ht]
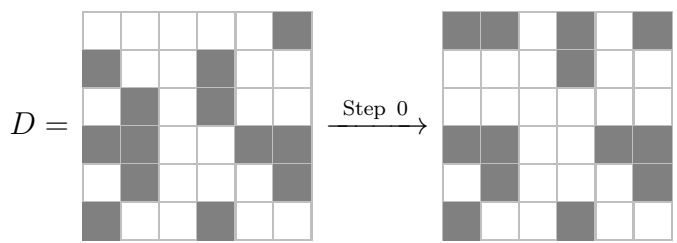

    \centering
    \vspace{-0.3cm}
  $$D=\BPD[1.2pc]{
    \O\O\O\O\O\gg\\
    \gg\O\O\gg\O\O\\
    \O\gg\O\gg\O\O\\
    \gg\gg\O\O\gg\gg\\
    \O\gg\O\O\O\gg\\
    \gg\O\O\gg\O\O\\
    }\xrightarrow{\mathrm{~Step~0~}}
    \BPD[1.2pc]{
    \gg\gg\O\gg\O\gg\\
    \O\O\O\gg\O\O\\
    \O\O\O\O\O\O\\
    \gg\gg\O\O\gg\gg\\
    \O\gg\O\O\O\gg\\
    \gg\O\O\gg\O\O\\
    }.$$
\vspace{-0.3cm}
 \caption{An illustration for Step 0 in Algorithm 2.}
   \label{Step-0}
\end{figure}

Notice that after Step 0, for any $1\leq j_1<j_2\leq n$, the boxes $(i_1-m, j_1)$ and $(i_1-m, j_2)$ for every $1\leq m\leq r_2(D;i_1, i_2; j_1, j_2)$ are blank boxes.

Step 1: Locate the leftmost column index, say column $j$, such that 
\begin{itemize}
    \item[(1)] the two boxes $(i_1, j)$ and $(i_2, j)$ form the configuration $
\BPD{\g\\
\O}
$;

    \item[(2)]  there exists (at least) one column index  $j'$ such that 
(i) both    $(i_1, j')$ and $(i_2, j')$ belong to $D$, and (ii)
$r_2(D;i_1,i_2; j,j')>0$ if $j<j'$, or $r_2(D;i_1,i_2; j',j)>0$ if $j'<j$.
\end{itemize}
We call the box $(i_1, j)$ a pivot. For the right picture in Figure \ref{Step-0}, the pivot box is $(4,1)$. Set $m:=1$.  If there is no pivot box, then the algorithm is done, and output $S_{i_1,i_2}(D)$.

Step 2: Locate all the column indices  $j'$ such that 
(i) both    $(i_1, j')$ and $(i_2, j')$ belong to $D$, and (ii)
$r_2(D;i_1,i_2; j,j')\ge m$ if $j<j'$, or $r_2(D;i_1,i_2; j',j)\ge m$ if $j'<j$. 
Suppose that there are $k$ such column indices, say  $j_1<\cdots<j_k$. For   $1\leq t\leq k$, let $C^t$ be the diagram obtained by moving the boxes $(i_2, j_1),\ldots, (i_2, j_t)$ up along the columns  to the positions $(i_1-m, j_1),\ldots, (i_1-m, j_t)$.  Set 
\[
S_{i_1,i_2}(D):=S_{i_1,i_2}(D)\cup \{C^1,\ldots, C^k\}.
\]

For the pivot $(4,1)$, we have $k=2$ and  $(j_1,j_2)=(2,6)$. The two diagrams $C^1$ and $C^2$ generated in this step are listed  in Figure  \ref{Step-22}.
\begin{figure}[ht]
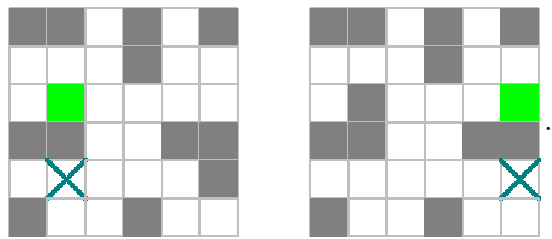

    \centering
    \vspace{-0.3cm}
 $$\BPD[1.2pc]{
    \gg\gg\O\gg\O\gg\\
    \O\O\O\gg\O\O\\
    \O\R\O\O\O\O\\
    \gg\gg\O\O\gg\gg\\
    \O\C\O\O\O\gg\\
    \gg\O\O\gg\O\O\\
    }\qquad
    \BPD[1.2pc]{
    \gg\gg\O\gg\O\gg\\
    \O\O\O\gg\O\O\\
    \O\gg\O\O\O\R\\
    \gg\gg\O\O\gg\gg\\
    \O\O\O\O\O\C\\
    \gg\O\O\gg\O\O\\
    }.$$
    \vspace{-0.3cm}
 \caption{An illustration for Step 2 in Algorithm 2.}
   \label{Step-22}
\end{figure}

Step 3: Place  all the $k$ boxes moved in Step 2 back to their original positions in row $i_2$ (thus recovering the diagram considered in Step 1). 
If there is some 
$r_2(D;i_1,i_2; j,j')> m$ for $j<j'$, or $r_2(D;i_1,i_2; j',j)> m$ for $j'<j$, then set 
$m:=m+1$ and  return  back to Step 2. Otherwise, move the pivot $(i_1,j)$ up by one unit to the position $(i_1-1,j)$,  and return  to Step 1. 

Continue the above example. 
For the pivot $(4,1)$, we have
\[
r_2(D;4,5; 1, 2)=r_2(D;4,5; 1, 6)=1.
\]
So, in Step 3, we do not need to go back to Step 2. What we do is to move the pivot  $(4,1)$ up to the position $(3,1)$ and then turn back to Step 1. Therefore,  the diagram, that will be considered in Step 1 in the second round of iteration,  is as depicted in 
Figure \ref{Step-22-45}.
\begin{figure}[ht]
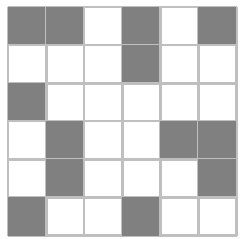

    \centering
    \vspace{-0.3cm}
$$
          \BPD[1.2pc]{
    \gg\gg\O\gg\O\gg\\
    \O\O\O\gg\O\O\\
    \gg\O\O\O\O\O\\
    \O\gg\O\O\gg\gg\\
    \O\gg\O\O\O\gg\\
    \gg\O\O\gg\O\O\\}
$$
  \vspace{-0.3cm}
 \caption{The diagram by moving the pivot up by a unit.}
   \label{Step-22-45}
\end{figure}

We move on to  the second round of iteration, starting with Step 1. Now the pivot box  is $(4,5)$. Notice that 
\[
r_2(D;4,5;2,5)=r_2(D;4,5;5,6)=2.
\]
So, after Step 2, we get two diagrams, as given in Figure \ref{Step-22-4521}.
\begin{figure}[ht]
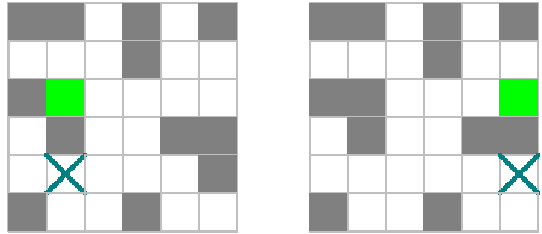

    \centering
    \vspace{-0.3cm}
$$
    \BPD[1.2pc]{
    \gg\gg\O\gg\O\gg\\
    \O\O\O\gg\O\O\\
    \gg\R\O\O\O\O\\
    \O\gg\O\O\gg\gg\\
    \O\C\O\O\O\gg\\
    \gg\O\O\gg\O\O\\}\qquad
    \BPD[1.2pc]{
    \gg\gg\O\gg\O\gg\\
    \O\O\O\gg\O\O\\
    \gg\gg\O\O\O\R\\
    \O\gg\O\O\gg\gg\\
    \O\O\O\O\O\C\\
    \gg\O\O\gg\O\O\\}
$$
 \vspace{-0.3cm}
 \caption{The first two diagrams generated in the second round of iteration.}
   \label{Step-22-4521}
\end{figure}
This time, in Step 3, we need to skip back to Step 2, generating  two more diagrams as shown  in Figure \ref{Step-2232-4521}.
\begin{figure}[ht]
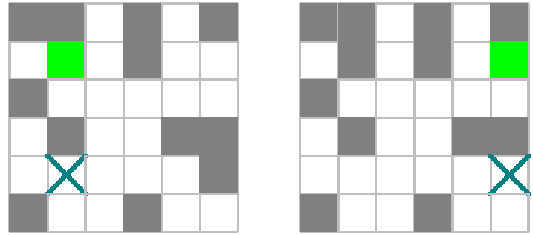

    \centering
    \vspace{-0.3cm}
$$\BPD[1.2pc]{
    \gg\gg\O\gg\O\gg\\
    \O\R\O\gg\O\O\\
    \gg\O\O\O\O\O\\
    \O\gg\O\O\gg\gg\\
    \O\C\O\O\O\gg\\
    \gg\O\O\gg\O\O\\
    }~~~~~ \BPD[1.2pc]{
    \gg\gg\O\gg\O\gg\\
    \O\gg\O\gg\O\R\\
    \gg\O\O\O\O\O\\
    \O\gg\O\O\gg\gg\\
    \O\O\O\O\O\C\\
    \gg\O\O\gg\O\O\\
    }.$$
  \vspace{-0.3cm}
 \caption{The second  two diagrams generated in the second round of iteration.}
   \label{Step-2232-4521}
\end{figure}

In the third round of iteration, there is no pivot any more, and so the algorithm terminates. Hence, for the diagram $D$ in Figure \ref{Step-0} and $(i_1, i_2)=(4,5)$, Algorithm 2 produces   a set $S_{4,5}(D)$ consisting of six diagrams less than $D$, as illustrated  in Figures \ref{Step-22}, \ref{Step-22-4521} and \ref{Step-2232-4521}.

We next show that after applying Algorithm 2 to all pairs $1< i_1<i_2\leq n $, the resulting diagrams have distinct weights. Moreover, these diagrams  have different weights from the diagrams generated by Algorithm 1. 

\begin{prop}\label{a2a2}
Let $1< i_1<i_2\leq n $. Then the diagrams in $S_{i_1,i_2}(D)$ have distinct weight vectors. Moreover, for $1<i_1'<i_2'\leq n$ with  $(i_1', i_2')\neq (i_1, i_2)$, the diagrams  in  $S_{i_1,i_2}(D)$ and   $S_{i_1',i_2'}(D)$ have distinct weight vectors.
\end{prop}

\begin{proof}
Let $C$ be any given diagram in $S_{i_1,i_2}(D)$, and let  $C'\in S_{i_1,i_2}(D)$ with $C'\neq C$ and  
$C''\in S_{i_1',i_2'}(D)$.
We first check that $C$ and $C'$ have different weights. Suppose that $C$ is generated  prior to $C'$. The arguments are divided into two cases. 

Case 1. $C$ and $C'$ correspond to the same pivot. By the construction of Algorithm 2, $C$ and $C'$ are obtained from a same diagram by moving some boxes in row $i_2$ up to the same row or two different rows. Keep in mind that $C$ appears earlier than $C'$. If the boxes are moved up to the same row, say row $i$, then $C$ has less boxes than $C'$ in row $i$. If the boxes are moved up to two different rows, then it is clear that $C$ and $C'$ have different numbers of boxes in these two rows.  In both situations, $C$ and $C'$ have distinct weights.

Case 2. $C$ and $C'$ correspond to distinct pivots. In this case, notice that $C$ has  less  boxes than $C'$ in row $i_1$. So $C$ and $C'$ have distinct weights. 

We next check that $C$ and $C''$ have distinct weights. We also have two cases.

Case 1':  $i_1\neq i_1'$. Assume  $i_1< i_1'$ without loss of generality. Let 
$$
k=\#\left\{1\leq j\leq n\colon r_1(D;i_1,j)> 0\right\}.
$$
Concerning $C''$, Step 0 moves all boxes of $D$ above row $i_1'$ to the topmost rows, and so one moves exactly $k$ boxes in row $i_1$ to higher rows. However, by the construction of Algorithm 2 applied to the pair $(i_1, i_2)$, it is not hard to observe  that there are at most $k-1$ boxes of $D$ (lying in row  $i_1$ or $i_2$) that are moved up to rows higher than row $i_1$ (because the last pivot in row $i_1$ will not be moved up). So the weights of $C$ and $C''$ cannot be the same.

Case 2': $i_1=i_1'$, but $i_2\neq i_2'$. 
Suppose that $i_2<i_2'$. Then $C''$ has less boxes than $C$
in row $i_2'$. So $C$ and $C''$ have distinct weights. This completes the proof.
\end{proof}

\begin{prop}
For $1<i_1<i_2\leq n$, the diagrams in $S_{i_1,i_2}(D)$  and the diagrams generated by Algorithm 1 have distinct weight vectors.
\end{prop}

\begin{proof}
Let $C\in S_{i_1,i_2}(D)$, and $C'$ be any diagram generated  by Algorithm 1. Suppose to the contrary that $C$ and $C'$ have the same weight vector. Note that $C$ has less boxes than $D$ in row $i_2$. Thus some boxes of $D$ in row $i_2$ must be moved up to form $C'$. This meanwhile tells that in the construction of $C'$, the boxes of $D$ lying above row $i_2$ are moved up to the topmost positions. Particularly, there are $k$ boxes of $D$ in row  $i_1$ that are moved up to higher rows, where, as in Case 1' in Proposition \ref{a2a2}, 
$$
k=\#\left\{1\leq j\leq n\colon r_1(D;i_1,j)\neq 0\right\}.
$$
However,  as explained  in Case 1' in Proposition \ref{a2a2}, there are at most $k-1$ boxes of $D$, that are moved up to the places above row $i_1$, to form 
$C$. This implies  that $C'$ has more boxes above row $i_1$ than $C$, leading to a contradiction.  
\end{proof}

\subsection{The third algorithm}\label{Sub-33}

We lastly describe the third  algorithm which will produce  $r_3$ diagrams with  weights different from those produced by Algorithm  1 or Algorithm 2. 


Denote by $D_{> i}$ the subdiagram of $D$ which includes the boxes of $D$ below row $i$. 
A box $(i,j)$ of $D$ is called  a PIVOT if $r_1(D; i,j)>0$ (equivalently, there is at least one blank box above $(i,j)$ in the same column). Here we use capital letters to distinguish with the pivots defined in Subsection  \ref{VHIHI-1}.

\begin{lem}
We have 
\[
r_3=\sum_{(i,j)}r_1(D; i,j)\times r_1(D_{>i}),
\]
where the sum ranges over all PIVOTs of $D$, and  $r_1(D_{>i})$ is the number of subdiagrams of $D_{>i}$ which are equal to the configuration (A) in Figure  \ref{config}.
\end{lem}

\begin{proof}
Consider the   subdiagrams  of $D$  which is equal to the configuration (C), or (C'), or (C'') in Figure  \ref{config}. Note that the  second box in each subdiagram  is a PIVOT. Given a PIVOT $(i,j)$, the subdiagrams, whose second box is $(i,j)$, are  
 counted by $r_1(D; i,j)\times r_1(D_{>i})$, and so the proof is  complete. 
\end{proof}

\textbf{Algorithm 3.} 
Fix a PIVOT box $(i,j)\in D$. This algorithm will produce $r_1(D; i,j)\times r_1(D_{>i})$ diagrams less than $D$.

Step 0.   Move the boxes of $D$ above row $i_1$, along with the boxes of $D$  in  row $i_1$ to the left of $(i,j)$, up along the column such that they occupy the positions in the topmost rows.  See Figure   \ref{Step-0-00} for an illustration. Notice that in the resulting diagram,  there are $r_1(D; i,j)$ blank boxes right above $(i,j)$.
\begin{figure}[ht]
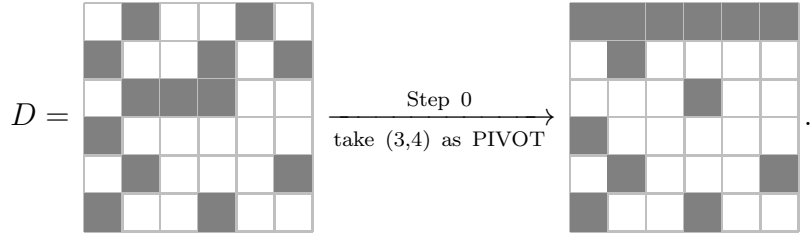

    \centering
    \vspace{-0.3cm}
  $$D=\BPD[1.2pc]{
    \O\gg\O\O\gg\O\\
    \gg\O\O\gg\O\gg\\
    \O\gg\gg\gg\O\O\\
    \gg\O\O\O\O\O\\
    \O\gg\O\O\O\gg\\
    \gg\O\O\gg\O\O\\
    }\xrightarrow[\mathrm{take~(3,4)~as~PIVOT}]{\mathrm{~Step~0~}}
    \BPD[1.2pc]{
    \gg\gg\gg\gg\gg\gg\\
    \O\gg\O\O\O\O\\
    \O\O\O\gg\O\O\\
    \gg\O\O\O\O\O\\
    \O\gg\O\O\O\gg\\
    \gg\O\O\gg\O\O\\
    }.$$
\vspace{-0.3cm}
 \caption{An illustration for Step 0 in Algorithm 3.}
   \label{Step-0-00}
\end{figure}

Step 1. Apply Algorithm 1 to $D_{>i}$. 
This gives rise to a chain of $r_1(D_{>i})$ diagrams less than $D$, say $C^1<C^2<\cdots<C^{k}$, where $k=r_1(D_{>i})$.
For $1\leq p\leq k$ and  $1\leq q\leq r_1(D; i,j)$, let $C^{p,q}$ be the diagram obtained from $C^p$ by moving the PIVOT $(i,j)$ up to the position $(i-q+1,j)$. 
This is well defined since there are $r_1(D; i,j)$ blank boxes right above $(i,j)$.
Note that $C^{p,1}=C^p$.
Set  
\[
T_{i,j}(D)= \left\{C^{p,q}\colon 1\leq p\leq k,\, 1\leq q\leq r_1(D; i,j)\right\}.
\]

\begin{prop}\label{a3a3}
Running over all PIVOTs of $D$, the  diagrams generated by Algorithm 3 have distinct weight vectors. 
\end{prop}

\begin{proof}
Assume that $C\in T_{i,j}(D)$ and $C'\in T_{i', j'}(D)$.  We show that   $C$ and $C'$  have distinct weights. We  consider the following two cases separately.

Case 1:   $(i,j)=(i',j')$. This can be seen as follows.
For $1\leq q \leq r_1(D; i,j)$, notice that 
\[
C^{1,q}<C^{2,q}<\cdots<C^{k,q},
\]
where   $k=r_1(D_{>i})$   as in Step 1. As explained in Proposition \ref{OBNHG}, we can conclude that $C^{1,q},\ldots,C^{k,q}$ have distinct weights.  
Moreover, by the construction in Step 1, it is obvious that the weights of $C^{p, q}$ and $C^{p', q'}$ for $p\neq p'$
are distinct. This verifies the assertion for $(i,j)=(i',j')$.

Case 2: $(i,j)\neq (i',j')$. Without loss of generality, assume that $i<i'$ or $i=i',j<j'$. We assert  that in row $i-r_1(D; i,j)$, $C'$ has more boxes than $C$. To see this, a key observation is that when Step 0 in  Algorithm 3 is applied  to the PIVOT $(i',j')$, the box $(i,j)$ is moved up to the position $(i-r_1(D; i,j),j)$. However, when Algorithm 3 is applied to the PIVOT $(i,j)$, Step 1 contributes no box to row 
$i-r_1(D; i,j)$ since  the box $(i,j)$  can be moved up at most  to the position  $(i-r_1(D; i,j)+1,j)$ (that is,   $(i-r_1(D; i,j),j)$ is a blank box in $C$). This implies that $C'$ has at least one more box  than $C$ in row $i-r_1(D; i,j)$, and so $C$ and $C'$ have distinct weights.
\end{proof}

\begin{prop}\label{a3a1}
For any PIVOT $(i,j)$, the  diagrams in $T_{i,j}(D)$ and the diagrams generated by Algorithm 1 have distinct weight vectors.
\end{prop}

\begin{proof}
Let $C\in T_{i,j}(D)$.
Write 
\[
C^1<C^2<\cdots <C^{r_1}<D
\]
for the chain of diagrams produced by Algorithm 1. Let $1\leq \ell\leq r_1$ be the index such that $C^\ell$ is the diagram obtained from its preceding diagram $C^{\ell+1}$ by moving $(i,j)$ up to the position $(i-1,j)$ (here we set $C^{r_1+1}=D$).
Notice  that 
\begin{itemize}
    \item $C^{\ell +1}$ is exactly the resulting diagram after applying Step 0 in Algorithm 3 to $D$;

    \item for $1\leq s< r_1(D;i,j)$, the diagram $C^{\ell+1-s}$ is obtained from $C^{\ell+1}$ by moving $(i,j)$ up to the position $(i-s,j)$
.
\end{itemize}
The arguments are divided into two cases.

Case 1.  $\ell+1 -r_1(D; i,j)<t\leq r_1$. In this case, note that $(C^t)_{>i}=D_{>i}$. However, $C_{>i}$ and $D_{>i}$ have different weights, and so the weights of $C$ and $C^t$
are distinct. 

Case 2. $1\leq t\leq \ell+1-r_1(D; i,j)$. In this case,  during the construction of $C^t$, the box $(i,j)$ of $D$ is moved up to the position $(i-r_1(D; i,j),j)$. 
For the analogous  reason to the proof of the $(i,j)\neq (i',j')$ case in  Proposition \ref{a3a3},  $C^t$ has more boxes than $C$ in row $i-r_1(D; i,j)$. This completes the proof. 
\end{proof}

\begin{prop}

For any PIVOT $(i,j)$ and any row indices $i_1<i_2$, the  diagrams in $T_{i,j}(D)$ and the diagrams in $S_{i_1,i_2}(D)$ generated by Algorithm 2 have distinct weight vectors.
\end{prop}

\begin{proof}
Let $C\in  T_{i,j}(D)$ and $C'\in S_{i_1,i_2}(D)$. We have two cases.

    Case 1: $i\neq i_1$. If $i<i_1$, then we can show that $C'$ has at least one more box than $C$ in row $i-r_1(D;i,j)$. The arguments are completely similar to   Case 2 in the proof of Proposition  \ref{a3a3}, and so is omitted. 
If $i>i_1$, then, as in  Case 1' in the proof of  Proposition  \ref{a2a2}, let 
$$
k=\#\left\{1\leq j\leq n\colon r_1(D;i_1,j)\neq 0\right\}.
$$
In the construction of  $C$, there are $k$ boxes of $D$  
in row $i_1$
that are moved up to the area higher than row $i_1$. While, in the construction of  $C'$, there are at most $k-1$ boxes (from row $i_1$ or $i_2$) of $D$ that are moved up to the area higher than row $i_1$. 
So we see that  $C$ has more boxes than  $C'$ in the area above row $i_1$, and thus $C$ and $C'$ have distinct weights.

    Case 2: $i=i_1$. 
    In this case, note that $C_{>i}$ has the same number of boxes as $D_{>i}$, while  $(C')_{>i}$ has less boxes than $D_{>i}$. This implies that $C$ and $C'$ must have distinct weights.   So the proof is complete. 
\end{proof}

To conclude this section, we see that  Algorithms 1, 2 and 3 could be used to produce  a total of $r_1+r_2+r_3$ diagrams that are less than $D$.  Combined with  the propositions  in Subsections \ref{Sub-11}, \ref{VHIHI-1} and  \ref{Sub-33}, these diagrams  have   distinct weight vectors. This provides a  proof of Theorem \ref{res1}.

\section{Proofs of Theorems \ref{schres} and \ref{forkey}} \label{paoe}

In this section, we specialize $D$ in Theorem \ref{res1} to a Rothe diagram or a skyline diagram, thereby completing   the proofs of Theorems \ref{schres} and \ref{forkey}. 

\subsection{Proof of Theorem \ref{schres}}
In this subsection, we let $D=D(w)$ be the Rothe diagram of a permutation $w\in S_n$. 
It suffices to  prove  the following.

\begin{theo}\label{0ing-1}
For $w\in S_n$, we have  \begin{align}\label{INYFG-1}
r_1(D(w))+r_2(D(w))&+r_3(D(w))\nonumber\\
        &\ge  p_{132}(w)+p_{1432}(w)+p_{13254}(w)+3p_{14253}(w)
        \nonumber\\
        &\ \ \ \ +p_{14352}(w)+4
p_{15243}(w)+p_{15324}(w)+2p_{15342}(w)
        \\
        &\ \ \ \ +p_{15432}(w)+p_{24153}(w)+2p_{25143}(w)+p_{35142}(w).\nonumber
\end{align}   
\end{theo}

\begin{proof}
For each permutation  $u$  appearing on the right-hand side of \eqref{INYFG-1}, let $P_u(w)$ denote the set of $u$ patterns in $w$, and $a_u$ be the coefficient of $p_u(w)$. For every $u$ pattern in $P_u(w)$, we shall construct $a_u$ subdiagams of $D(w)$, each of  which is equal to one of the configurations in Figure \ref{config}.  Then we conclude the proof by explaining that such constructed subdiagrams are  different from each other.

We first look at the construction of subdiagrams of $D(w)$ for $u=132$ or  $1432$. 
Let $w(i_1)w(i_2)w(i_3)\in P_{132}(w)$. The subdiagram of $D(w)$ generated by $w(i_1)w(i_2)w(i_3)$ is defined as   $\{(i_1, w(i_2)), (i_2, w(i_2))\}$, which is  the configuration (A) in  Figure \ref{config}. 
The construction is  displayed in the first line of Table \ref{tab:my_label-1}, where the boxes forming  the subdiagram are marked with   {\color{black} $\checkmark$}. 
This correspondence  has appeared in the proof of \cite[Corollary 19]{meszaros2021principal}, which is in fact a bijection between $P_{132}(w)$ and the set of subdiagrams of $D(w)$ which are equal to the configuration (A) in  Figure \ref{config}.
\begin{table}[h]
    \centering

    \begin{tabular}{c|cccc}
 patterns& subdiagrams& & & \\  
\hline
$    \BPD[1.2pc]{
    \N{}\\
    \T\H\H\\
    \I\gg\T\\
    \I\T\X\\
    \N{}\N{\text{\small 132 pattern}}}$&$    \BPD[1.2pc]{
    \N{}\\
    \T\HC\H\\
    \I\ggC\T\\
    \I\T\X\\
    \N{}\N{\text{\footnotesize (A)}}}$\\
    \hline
    $    \BPD[1.2pc]{
    \N{}\\
    \T\H\H\H\\
    \I\gg\gg\T\\
    \I\gg\T\X\\
    \I\T\X\X\\
    \N{}\N{\text{\small ~~ 1432 pattern}}}$&
    $    \BPD[1.2pc]{
    \N{}\\
    \T\HC\HC\H\\
    \I\ggC\ggC\T\\
    \I\ggC\TC\X\\
    \I\T\X\X\\
    \N{}\N{~~~~\text{\footnotesize (B)}}}$\\
\end{tabular}

    \caption{Subdiagrms generated by  132 or 1432 patterns.}
    \label{tab:my_label-1}
\end{table}

We next turn to the pattern $w(i_1)w(i_2)w(i_3)w(i_4)\in P_{1432}(w)$. The subdiagram of $D(w)$ generated by this pattern is defined as 
\[\left\{(i_1, w(i_4)), (i_1, w(i_3)),(i_2, w(i_4)), (i_2, w(i_3)), (i_3, w(i_4)), (i_3, w(i_3))\right\},\]
which forms the configuration (B) in   Figure \ref{config}. 
See the second  line of Table \ref{tab:my_label-1} for an illustration of the construction.

There are 10 permutations  $u$ in $S_5$ appearing on the right-hand side of \eqref{INYFG-1}. We put the constructions of subdiagrams of $D(w)$ for these $u$ patterns in Tables \ref{table-09} and \ref{tab:my_label-98} in the appendix. 

Let $\mathrm{Sub}(D(w))$ denote the collection (as multiset)  of all subdiagrams of $D(w)$ which could be produced  by   the $u$ patterns of $w$ as displayed in Tables \ref{tab:my_label-1}, \ref{table-09} and  \ref{tab:my_label-98}. 
The remaining work is to check that the subdiagrams  in $\mathrm{Sub}(D(w))$   are different. 
To do this, a crucial feature  that we can observe  from  Tables \ref{tab:my_label-1}, \ref{table-09} and  \ref{tab:my_label-98}  is that for any given subdaigram, say $D_{\mathrm{sub}}$, in $\mathrm{Sub}(D(w))$, we are able to recover  the (unique) pattern in $w$ from which $D_{\mathrm{sub}}$ is generated. 
To be specific, we  have the following explanations.
\begin{itemize}
    \item $D_{\mathrm{sub}}$ is the configuration (A) in Figure \ref{config}. In this case, $D_{\mathrm{sub}}$ has two boxes. Assume that the boxes lie in rows $\{i_1<i_2\}$ and column $j$. Then  the corresponding  132 pattern of $w$ includes the entries of $w$ at the positions $\{i_1, i_2, w^{-1}(j)\}$, where $w^{-1}$ is the inverse of $w$.  

    \item $D_{\mathrm{sub}}$ is the configuration (B), or (B') in Figure \ref{config}.
 Assume that the six boxes in $D_{\mathrm{sub}}$  lie in rows $\{i_1<i_2<i_3\}$ and columns $\{j_1<j_2\}$.  Then  the corresponding  pattern of $w$  includes the entries of $w$ at the positions $\{i_1, i_2, i_3, w^{-1}(j_1), w^{-1}(j_2)\}$.  
 It should be noted that there may happen that $i_3=w^{-1}(j_2)$, and in this case the pattern is a 1432 pattern.

    \item $D_{\mathrm{sub}}$ is the configuration (C), or (C'), or (C'') in Figure \ref{config}. Assume that the four boxes in $D_{\mathrm{sub}}$ lie in rows  $\{i_1<i_2<i_3<i_4\}$, and among the four boxes, the  lowest two     lie in column $j$. 
 Then  the corresponding pattern of $w$  includes the entries of $w$ at the positions $\{i_1, i_2, i_3, i_4, w^{-1}(j)\}$.
\end{itemize}

In view of the above observations, give two subdiagrams, say $D_{\mathrm{sub}}^1$ and $D_{\mathrm{sub}}^2$, in $\mathrm{Sub}(D(w))$, we can verify $D_{\mathrm{sub}}^1\neq  D_{\mathrm{sub}}^2$  by  contradiction. Suppose otherwise that $D_{\mathrm{sub}}^1=D_{\mathrm{sub}}^2$. Then they are generated by the same $u$ pattern in $w$. This only possibly happens in the case of the $u=15342$ pattern in Table \ref{tab:my_label-98}. However, the two subdigrams generated by a $u=15342$ pattern  are obviously distinct. This arrives at a contradiction. So the proof is complete. 
\end{proof}

\begin{re}
We remark that there may possibly exist instances of subdiagrams of $D(w)$, which are equal to the configurations in Figure \ref{config}, but cannot be produced  by the patterns of $w$ listed in Tables \ref{tab:my_label-1}, \ref{table-09}, or \ref{tab:my_label-98}.  
For example, consider the Rothe diagram $D(w)$ for $w=162435$. 
\begin{figure}[h t]
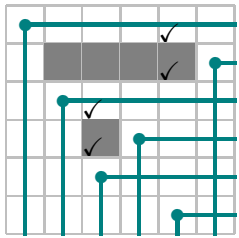

    \centering
$ \BPD[1.2pc]{
    \T\H\H\H\HC\H\\
    \I\gg\gg\gg\ggC\T\\
    \I\T\HC\H\H\X\\
    \I\I\ggC\T\H\X\\
    \I\I\T\X\H\X\\
    \I\I\I\I\T\X\\
    }   $
    \caption{A subdiagram of $D(162435)$.}
    \label{FIg:enter-label}
\end{figure}
Take the subdiagram consisting of the boxes with check marks. This subdiagram  cannot be generated by any $u\in S_5$ pattern of $w$. If it were generated by a $u$ pattern, then the $u$ pattern of $w$ would be $16243$ and so $u=15243$. However, in view of the  construction  in Table \ref{table-09}, the four subdiagrams of $D(162435)$ generated by the   pattern $16243$ do not include the one we are considering. 
\end{re}

\subsection{Proof of Theorem  \ref{forkey}}

Here we let $D=D(\alpha)$ in Theorem \ref{res1} be the skyline diagram of a weak composition $\alpha$. We   obtain a proof of  Theorem \ref{forkey} by verifying the following equality.

\begin{theo}
    For any (weak) composition $\alpha$, we have
\begin{align}\label{keyres}
r_1(D(\alpha))+r_2(D(\alpha))+r_3(D(\alpha))= & \sum_{\mathrm{inv}_1(\alpha)}(\alpha_{i_2}-\alpha_{i_1})+\sum_{\mathrm{inv}_2(\alpha)}(\alpha_{i_2}-\alpha_{i_3})\cdot(\alpha_{i_3}-\alpha_{i_1})\nonumber\\[5pt] & \ \ +\sum_{\mathrm{inv}_{3}(\alpha)}(\alpha_{i_2}-\alpha_{i_1})\cdot(\alpha_{i_4}-\alpha_{i_3}).
\end{align}   
\end{theo}

\begin{proof}
It is obvious that
\[
r_1(D(w))=\sum_{(i_1,i_2)\in \mathrm{inv}_1(\alpha)}(\alpha_{i_2}-\alpha_{i_1}).
\] 
We next check that 
\begin{align}\label{fg-09}
  r_2(D(w))=\sum_{(i_1, i_2, i_3)\atop
\in \mathrm{inv}_2(\alpha)}(\alpha_{i_2}-\alpha_{i_3})\cdot(\alpha_{i_3}-\alpha_{i_1}).  
\end{align}
This can be seen as follows. First,  since $D(\alpha)$ is left-justified, there is no subdiagram of $D(\alpha)$ which is the configuration (B'). We are now  left with the enumeration of    subdigrams of $D(\alpha)$ which are equal to the configuration (B). Compute such  subdiagrams which lie in rows $\{i_1<i_2<i_3\}$. Clearly, we have $\alpha_{i_1}<\alpha_{i_3}<\alpha_{i_2}$. 
Observe that for each such subdiagram, the column index of its left three boxes is greater than $\alpha_{i_1}$ but less than or equal to $\alpha_{i_3}$, while the column  index of its right three boxes  is greater than $\alpha_{i_3}$ but less than or equal to $\alpha_{i_2}$. So there are a total of  $(\alpha_{i_2}-\alpha_{i_3})\cdot(\alpha_{i_3}-\alpha_{i_1})$ subdaigrams, which are the  configuration  (B), lying in rows  $\{i_1<i_2<i_3\}$.  This justifies  \eqref{fg-09}.

Using similar analysis, one can readily verify that 
$$r_3(D(\alpha))=\sum_{(i_1, i_2, i_3, i_4)\atop
\in \mathrm{inv}_3(\alpha)}(\alpha_{i_2}-\alpha_{i_1})\cdot(\alpha_{i_4}-\alpha_{i_3}).$$
This completes the proof of \eqref{keyres}.    
\end{proof}

\section{Concluding remarks}\label{condlu-o9}

In this section, we investigate some conjectures and problems concerning  $\theta_w$.

\subsection{The maximum value of $\theta_w$}

Let $\alpha_n$ be the largest    principal specialization  for  Schubert
 polynomials:
\[
\alpha_n=\max\{\nu_w\colon w\in S_n\}
\]
Merzon and  Smirnov  \cite{MS-1} predicted that the maximum value $\alpha_n$ is achieved on layered permutations. 
For positive integers $b_1,\ldots, b_k$ summing up to $n$, the associated layered permutation $w(b_1,\ldots, b_k)$ in $S_n$ is defined as the concatenation $w^1\cdots w^k$ of $k$ words, where  $w^i$ is obtained  by permuting  the  entries in the interval $[b_1+\cdots+b_{i-1}+1, b_1+\cdots+b_i]$ decreasingly. Here we set $b_0=0$. For example, $w(2,3,2,1)=21 543 76 8$.

\begin{conj}[Merzon--Smirnov\cite{MS-1}]
    For $n\geq 1$, the permutations in $S_n$ attaining $\alpha_n$ are layered permutations. 
\end{conj}

Consider the largest value of $\theta_w$ with $w\in S_n$: 
\[
\beta_n=\max\{\theta_w\colon w\in S_n\}
\]
Computer evidence suggests  that   $\beta_n$ is also achieved  on layered permutations.

 \begin{conj}\label{conk-09}
    For $n\geq 1$, the permutations in $S_n$ attaining the maximum value $\beta_n$ are layered permutations. 
\end{conj}

Conjecture  \ref{conk-09} has been verified   for $n$ up to $9$. Note that this conjecture is equivalent to saying that  the permutations in $S_n$ whose Newton polytopes have the largest number of lattice points are layered permutations.
In Table \ref{tab:my_label-0o}, we list the values of  $\alpha_n$ and $\beta_n$ for $n\leq 9$, together with the permutations achieving these maximum values. 
\begin{table}[h t]
    \centering
    \begin{tabular}{c|c|c|c|c}
       $n$& $\alpha_n$ & \text{permutations attaining $\alpha_n$} & $\beta_n$ & \text{permutations attaining $\beta_n$}  \\
       \hline
        1 & 1 & 1 & 1 & 1\\
        2 & 1 & 12, 21 & 1 & 12, 21\\
        3 & 2 & 132   & 2 & 132\\
        4 & 5 & 1432  & 5 & 1432\\
        5 & 14 & 12543, 15432, 21543 & 14 & 15432\\
        6 & 84 & 126543, 216543 & 65 & 126543, 216543\\
        7 & 660 & 1327654 & 347 & 1276543, 2176543\\
        8 & 9438 & 13287654 & 2151 &13287654\\
        9 & 163592 & 132987654 & 17319 & 132987654\\
    \end{tabular}
    \caption{The values of $\alpha_n$ and $\beta_n$ for $n\leq 9$.}
    \label{tab:my_label-0o}
\end{table}
From the table, we see that for each $n=1,2,\ldots, 9$ expect for $n=7$, there   exists (at least) one common layered permutation   which reaches both $\alpha_n$ and $\beta_n$. 

\begin{prob}
For $n$ large, does there always exist  a (layered) permutation in $S_n$ that achieves the maximum values $\alpha_n$ and $\beta_n$ simultaneously.    
\end{prob}

The asymptotic behavior of $\alpha_n$ was first  sought by 
Stanley \cite{stanley}.

\begin{conj}[Stanely \cite{stanley}]\label{Sju8}
 There exists a limit 
 \[
 \lim_{n \rightarrow \infty} \frac{\log \alpha_n}{n^2}.
 \]
\end{conj}

Morales,  Pak  and   Panova \cite{MPP-1} showed that there is a limit when restricted to layered permutations. That is, letting 
\[
\gamma_n=\max\{\nu_w\colon \text{$w$ are layered permutations in $S_n$}\},
\]
there exists a limit
\[
 \lim_{n \rightarrow \infty} \frac{\log_2 \gamma_n}{n^2}\approx 0.2932362762.
 \]
Zhang \cite{Zhang-1} recently considered the asymptotic property  for the largest value of 
$\S_w(x)|_{x_i={q}^{i-1}}$  with $q=-1$ for  multi-layered permutations.

\begin{prob}
Does there exist an asymptotic behavior for $\beta_n$ similar to Conjecture 
    \ref{Sju8}.
\end{prob}

\subsection{A positivity conjecture }

For $w\in S_n$,  we may write 
\[
\nu_w=1+\sum_{u\in S_m\atop
m\leq n} c_u\, p_{u}(w),
\]
where the coefficients $c_u$ for $u\in S_m$ are   determined recursively  by 
\[
c_u=\nu_u-1-\sum_{\sigma\in S_{\ell}\atop
\ell<m} c_\sigma\, p_{\sigma}(u).
\]

As observed by Gao \cite[Lemma 3.1]{gao}, the coefficients $c_u$ own   the stability property, that is, $c_u=0$ for $u\in S_m$ with $u(m)=m$. The following appealing  conjecture appears as   \cite[Conjecture 3.2]{gao}.

\begin{conj}[{Gao \cite{gao}}]
For any permutation $u$, we have $c_u\in \mathbb{Z}_{\geq 0}$.    
\end{conj}

The above conjecture  has been confirmed for permutations avoiding both 1432 and 1423 patterns by  M{\'e}sz{\'a}ros  and  Tanjaya \cite{MT-211}, and for permutations avoiding 1243 patterns by Dennin \cite{Denn}.

Similarly, for $w\in S_n$, one may express 
\[
\theta_w=1+\sum_{u} d_u\, p_{u}(w).
\]
The coefficients $d_u$ can be similarly computed in a recursive procedure:
\[
d_u=\theta_u-1-\sum_{\sigma\in S_\ell\atop
\ell<m} d_\sigma\, p_{\sigma}(u).
\]
Imitating the arguments  in the proof of \cite[Lemma 3.1]{gao}, we can show  that $d_u=0$ for $u\in S_m$ with $u(m)=m$.

\begin{conj}\label{locond-3}
For any permutation $u$, we have $d_u\in \mathbb{Z}_{\geq 0}$.    
\end{conj}

Conjecture \ref{locond-3} has been verified for $n$ up to $8$. The data imply that 
\begin{itemize}
    \item when $1\leq n\leq 5$, we have
$0\leq d_u\leq c_u$, and    $d_u>0$ if and only if $c_u>0$;

\item when $n=6,7$, we still have $0\leq d_u\leq c_u$. But there exist two permutations in $S_6$, 
$u^1=136245$ and $u^2=146235$, such that 
$d_{u^1}=d_{u^2}=0$ but $c_{u^1}=c_{u^2}=1$;

\item 
when $n=8$, we no  longer have $0\leq d_u\leq c_u$. The only exception is   $u=13452786$ for which  $c_u=3$  and  $d_u=4$.
\end{itemize}
In Table \ref{tab-mdu}, 
we list the  values of $c_u>0$ (or  $d_u> 0$) for permutations $u\in S_m$ with  $m\leq 5$, where the  permutations  appearing in the lower bound in Theorem  \ref{schres} are underlined.

\begin{table}[h t]
    \centering
\begin{tabular}{ccc|ccc}

  permutation & $c_u$ & $d_u$&permutation & $c_u$ & $d_u$ \\  
\hline
\underline{132} & 1 & 1 & \underline{1432} & 1 & 1\\
12453 & 1 & 1 & \underline{15342} & 2 & 2\\  
12534 & 1 & 1 & 15423 & 1 & 1\\  
12543 & 5 & 4 & \underline{15432} & 3 & 3\\  
\underline{13254} & 3 & 2 & 21453 & 1 & 1\\
13524 & 3 & 2 & 21534 & 1 & 1\\
13542 & 4 & 3 & 21543 & 5 & 4\\
\underline{14253} & 3 & 3 & \underline{24153} & 1 & 1\\
\underline{14352} & 1 & 1 & \underline{25143} & 2 & 2 \\
14523 & 1 & 1 & 31524 & 1 & 1\\
14532 & 1 & 1 & 31542 & 2 & 2\\
\underline{15243} & 4 & 4 & \underline{35142} & 1 & 1\\
\underline{15324} & 1 & 1 \\
\end{tabular}    
    \caption{Permutations in $S_m$ for $m\leq 5$ with nonzero  values of $c_u$ and $d_u$.}
    \label{tab-mdu}
\end{table}

From Table \ref{tab-mdu}, we see that there are 13 permutations  in $S_5$ that do not  appear in the lower bound in  Theorem   \ref{schres}. Among the 12 permutations appearing  in the lower bound, the coefficients for $p_{13254}(w)$ and $p_{15432}(w)$ are not optimal if assuming Conjecture \ref{locond-3}. New algorithms or tools are needed to explore for  further improving  the lower bound established in 
 Theorem   \ref{schres}.

\begin{prob}
Strengthen the lower bound in   Theorem   \ref{schres}.  
\end{prob}

As an attempt to enhance the bound, it would be interesting to establish a bound for $\theta_w$ encompassing   all $p_u(w)$ for $u$ being  permutations listed   in  Table \ref{tab-mdu}.

\newpage
\section{Appendix}\label{app-21}

\begin{table}[!h]
\centering
        \begin{tabular}{c|cccc}
 patterns& subdiagrams & & & \\  
\hline
$    \BPD[1.2pc]{
    \N{}\\
    \T\H\H\H\H\\
    \I\gg\T\H\H\\
    \I\T\X\H\H\\
    \I\I\I\gg\T\\
    \I\I\I\T\X\\
    \N{}\N{}\N{\text{\small 13254 pattern}}}$&$    \BPD[1.2pc]{
    \N{}\\
    \T\HC\H\H\H\\
    \I\ggC\T\H\H\\
    \I\T\X\HC\H\\
    \I\I\I\ggC\T\\
    \I\I\I\T\X\\
    \N{}\N{}\N{{\text{\footnotesize (C)}}}}$ \\
    \hline
    $    \BPD[1.2pc]{
    \N{}\\
    \T\H\H\H\H\\
    \I\gg\gg\T\H\\
    \I\T\H\X\H\\
    \I\I\gg\I\T\\
    \I\I\T\X\X\\    
     \N{}\N{}\N{\text{\small 14253 pattern}}}$&$\BPD[1.2pc]{
    \N{}\\
    \T\HC\HC\H\H\\
    \I\ggC\ggC\T\H\\
    \I\T\H\X\H\\
    \I\IC\ggC\I\T\\
    \I\I\T\X\X\\    
    \N{}\N{}\N{{\text{\footnotesize (B')}}}}$&$\BPD[1.2pc]{
    \N{}\\
    \T\HC\H\H\H\\
    \I\ggC\gg\T\H\\
    \I\T\HC\X\H\\
    \I\I\ggC\I\T\\
    \I\I\T\X\X\\    
    \N{}\N{}\N{{\text{\footnotesize (C)}}}}$&$\BPD[1.2pc]{
    \N{}\\
    \T\H\HC\H\H\\
    \I\gg\ggC\T\H\\
    \I\T\HC\X\H\\
    \I\I\ggC\I\T\\
    \I\I\T\X\X\\    
    \N{}\N{}\N{{\text{\footnotesize (C'')}}}}$\\
    \hline
     $\BPD[1.2pc]{\N{}\\
    \T\H\H\H\H\\
    \I\gg\gg\T\H\\
    \I\gg\T\X\H\\
    \I\gg\I\I\T\\
    \I\T\X\X\X\\
     \N{}\N{}\N{\text{\small 14352 pattern}}}$&
     $\BPD[1.2pc]{\N{}\\
    \T\HC\HC\H\H\\
    \I\ggC\ggC\T\H\\
    \I\gg\T\X\H\\
    \I\ggC\IC\I\T\\
    \I\T\X\X\X\\
    \N{}\N{}\N{{\text{\footnotesize (B)}}}}$\\
    \hline
    $\BPD[1.2pc]{\N{}\\
    \T\H\H\H\H\\
    \I\gg\gg\gg\T\\
    \I\T\H\H\X\\
    \I\I\gg\T\X\\
    \I\I\T\X\X\\
     \N{}\N{}\N{\text{\small 15243   pattern}}}$&
        $\BPD[1.2pc]{\N{}\\
    \T\HC\HC\H\H\\
    \I\ggC\ggC\gg\T\\
    \I\T\H\H\X\\
    \I\IC\ggC\T\X\\
    \I\I\T\X\X\\
    \N{}\N{}\N{{\text{\footnotesize (B')}}}}$&
        $\BPD[1.2pc]{\N{}\\
    \T\HC\H\H\H\\
    \I\ggC\gg\gg\T\\
    \I\T\HC\H\X\\
    \I\I\ggC\T\X\\
    \I\I\T\X\X\\
    \N{}\N{}\N{{\text{\footnotesize (C)}}}}$&
        $\BPD[1.2pc]{\N{}\\
    \T\H\H\HC\H\\
    \I\gg\gg\ggC\T\\
    \I\T\HC\H\X\\
    \I\I\ggC\T\X\\
    \I\I\T\X\X\\
    \N{}\N{}\N{{\text{\footnotesize (C')}}}}$&
        $\BPD[1.2pc]{\N{}\\
    \T\H\HC\H\H\\
    \I\gg\ggC\gg\T\\
    \I\T\HC\H\X\\
    \I\I\ggC\T\X\\
    \I\I\T\X\X\\
    \N{}\N{}\N{{\text{\footnotesize (C'')}}}}$\\
    \hline
    $\BPD[1.2pc]{\N{}\\
    \T\H\H\H\H\\
    \I\gg\gg\gg\T\\
    \I\gg\T\H\X\\
    \I\T\X\H\X\\
    \I\I\I\T\X\\
     \N{}\N{}\N{\text{\small 15324 pattern}}}$&
        $\BPD[1.2pc]{\N{}\\
    \T\HC\H\HC\H\\
    \I\ggC\gg\ggC\T\\
    \I\ggC\T\HC\X\\
    \I\T\X\H\X\\
    \I\I\I\T\X\\
    \N{}\N{}\N{{\text{\footnotesize (B)}}}}$\\

    \end{tabular}

     \caption{Subdiagrams in the proof of Theorem \ref{schres}  generated  by $u$ patterns with $u\in S_5$.}
     \label{table-09}
 \end{table}
 \newpage
\begin{table}[!h]
    \centering

  \begin{tabular}{c|cccc}
 patterns& subdiagrams & & & \\  
    \hline
    
    $    \BPD[1.2pc]{
    \N{}\\
    \T\H\H\H\H\\
    \I\gg\gg\gg\T\\
    \I\gg\T\H\X\\
    \I\gg\I\T\X\\
    \I\T\X\X\X\\
     \N{}\N{}\N{\text{\small 15342 pattern}}}$&
    $    \BPD[1.2pc]{
    \N{}\\
    \T\HC\HC\H\H\\
    \I\ggC\ggC\gg\T\\
    \I\gg\T\H\X\\
    \I\ggC\IC\T\X\\
    \I\T\X\X\X\\
    \N{}\N{}\N{{\text{\footnotesize (B)}}}}$&
    $    \BPD[1.2pc]{
    \N{}\\
    \T\HC\H\HC\H\\
    \I\ggC\gg\ggC\T\\
    \I\ggC\T\HC\X\\
    \I\gg\I\T\X\\
    \I\T\X\X\X\\
    \N{}\N{}\N{{\text{\footnotesize (B)}}}}$\\
   \hline

$
    \BPD[1.2pc]{
    \N{}\\
    \T\H\H\H\H\\
    \I\gg\gg\gg\T\\
    \I\gg\gg\T\X\\
    \I\gg\T\X\X\\
    \I\T\X\X\X\\
     \N{}\N{}\N{\text{\small 15432 pattern}}}
    $&$
    \BPD[1.2pc]{
    \N{}\\
    \T\HC\H\HC\H\\
    \I\ggC\gg\ggC\T\\
    \I\gg\gg\T\X\\
    \I\ggC\T\XC\X\\
    \I\T\X\X\X\\
    \N{}\N{}\N{{\text{\footnotesize (B)}}}}
    $\\
    \hline
$ \BPD[1.2pc]{
    \N{}\\
    \gg\T\H\H\H\\
    \gg\I\gg\T\H\\
    \T\X\H\X\H\\
    \I\I\gg\I\T\\
    \I\I\T\X\X\\
     \N{}\N{}\N{\text{\small 24153 pattern}}}   $&$ \BPD[1.2pc]{
    \N{}\\
    \gg\T\HC\H\H\\
    \gg\I\ggC\T\H\\
    \T\X\HC\X\H\\
    \I\I\ggC\I\T\\
    \I\I\T\X\X\\
    \N{}\N{}\N{{\text{\footnotesize (C'')}}}}   $\\
    \hline
        $\BPD[1.2pc]{
    \N{}\\
    \gg\T\H\H\H\\
    \gg\I\gg\gg\T\\
    \T\X\H\H\X\\
    \I\I\gg\T\X\\
    \I\I\T\X\X\\
     \N{}\N{}\N{\text{\small 25143 pattern}}}
    $&$\BPD[1.2pc]{
    \N{}\\
    \gg\T\H\HC\H\\
    \gg\I\gg\ggC\T\\
    \T\X\HC\H\X\\
    \I\I\ggC\T\X\\
    \I\I\T\X\X\\
    \N{}\N{}\N{{\text{\footnotesize (C')}}}}
    $&
    $\BPD[1.2pc]{
    \N{}\\
    \gg\T\HC\H\H\\
    \gg\I\ggC\gg\T\\
    \T\X\HC\H\X\\
    \I\I\ggC\T\X\\
    \I\I\T\X\X\\
    \N{}\N{}\N{{\text{\footnotesize (C'')}}}}
    $\\
    \hline
      $\BPD[1.2pc]{
    \N{}\\
    \gg\gg\T\H\H\\
    \gg\gg\I\gg\T\\
    \T\H\X\H\X\\
    \I\gg\I\T\X\\
    \I\T\X\X\X\\
     \N{}\N{}\N{\text{\small 35142 pattern}}}$&
      $\BPD[1.2pc]{
    \N{}\\
    \gg\gg\T\HC\H\\
    \gg\gg\I\ggC\T\\
    \T\HC\X\H\X\\
    \I\ggC\I\T\X\\
    \I\T\X\X\X\\
    \N{}\N{}\N{{\text{\footnotesize (C')}}}}$&$\BPD[1.2pc]{
    \N{}\N{}\N{}\N{}\N{}}$&$\BPD[1.2pc]{
    \N{}\N{}\N{}\N{}\N{}}$&$\BPD[1.2pc]{
    \N{}\N{}\N{}\N{}\N{}}$
\end{tabular}
    
    \caption{Subdiagrams in the proof of Theorem \ref{schres} generated  by $u$ patterns with $u\in S_5$ (continued).}
    \label{tab:my_label-98}
\end{table}

\footnotesize{

\textsc{(Peter L. Guo) Center for Combinatorics, Nankai University, LPMC, Tianjin 300071, P.R. China}

{\it
Email address: \tt lguo@nankai.edu.cn}

\medbreak

\textsc{(Zhuowei Lin) Center for Combinatorics, Nankai University, LPMC, Tianjin 300071, P.R. China}

{\it
Email address: \tt zwlin0825@163.com}


\begin{thebibliography}{1}

\bibitem{ARY}
A. Adve,  C. Robichaux  and A. Yong,  An efficient algorithm for deciding vanishing of Schubert polynomial coefficients,  Adv. Math. 383 (2021), Paper No. 107669, 38 pp.


\bibitem{D-1} M. Demazure, D\'esingularisation des vari\'et\'es de Schubert g\'en\'eralis\'ees, Ann. Sci. \'Ecole Norm. Sup. 7 (1974), 53--88.

\bibitem{D-2}
M. Demazure, Une nouvelle formule des caract\'eres, Bull. Sci. Math. 98 (1974),
 163--172.

\bibitem{Denn}
H. Dennin,
 Pattern bounds for principal specializations of $\beta$-Grothendieck polynomials, 	arXiv:2206.10017, 2022. 

\bibitem{2018Schubert}
A. Fink, K. M{\'e}sz{\'a}ros and A. St. Dizier,
Schubert polynomials as integer point transforms of generalized permutahedra,
Adv. Math. 332 (2018), 465--475.

\bibitem{fink2021zero}
A. Fink, K. M{\'e}sz{\'a}ros and A. St. Dizier,
Zero-one Schubert polynomials,
Math. Z. 297 (2021),   1023--1042.

\bibitem{FS}
S. Fomin and R.P. Stanley, Schubert polynomials and the nil-Coxeter algebra, Adv. Math. 103 (1994),  196--207.

\bibitem{gao}
Y. Gao,
Principal specializations of Schubert polynomials and pattern containment,
European J. Combin. 94 (2021), Paper No. 103291, 12 pp.

\bibitem{GLP}
P.L. Guo, Z. Lin and  S.C.Y. Peng,
Zero-one dual characters of flagged Weyl modules,
arXiv:2411.10933v1, 2024.



\bibitem{HPSW}
Z. Hamaker,  O. Pechenik, D.E. Speyer and A. Weigandt,  Derivatives of Schubert polynomials and proof of a determinant conjecture of Stanley,  Algebraic  Combin. 3 (2020), 301--307.

\bibitem{KP-1}
W. Kra{\'s}kiewicz and P. Pragacz,
Foncteurs de schubert,
C. R. Acad. Sci. Paris Sér. I Math. 304 (1987),  209--211.

\bibitem{KP-2}
W. Kra{\'s}kiewicz and P. Pragacz,
Schubert functors and Schubert polynomials,
European J. Combin. 25 (2004),  1327--1344.

\bibitem{schubert}
A. Lascoux, M.-P. Sch\"utzenberger,
Polyn\^omes de Schubert,
C. R. Acad. Sci. Paris Sér. I Math. 294 (1982),   447--450.

\bibitem{Mac}
I.G. Macdonald, Notes on Schubert Polynomials, Laboratoire de combinatoire et d'informatique math\'ematique (LACIM), Universit\'e du Qu\'ebec \'a Montr\'eal, Montreal, 1991.

\bibitem{magyar1998schubert}
P. Magyar,
Schubert polynomials and Bott--Samelson varieties,
Comment. Math. Helv. 73 (1998), 603--636.

\bibitem{MS-1}
G. Merzon and E. Smirnov, Determinantal identities for flagged Schur and Schubert
 polynomials, European J. Math.  2 (2016), 227--245.
 

\bibitem{meszaros2021principal}
K. M{\'e}sz{\'a}ros, A. St. Dizier and A. Tanjaya, 
Principal specialization of dual characters of flagged Weyl modules,
Electron.  J. Combin. 28 (2021),   Paper No. 4.17, 12 pp.


\bibitem{MT-211}
K. M{\'e}sz{\'a}ros  and A. Tanjaya, 
 Inclusion-exclusion on Schubert polynomials,
 Algebraic Combin. 5  (2022),  209--226.

\bibitem{MTY}
C. Monical, N. Tokcan and A. Yong, Newton polytopes in algebraic combinatorics,
 Selecta Math. (N.S.) 25 (2019), no. 5, Paper No. 66.

\bibitem{MPP-1}
A.H. Morales, I. Pak  and G. Panova, Asymptotics of principal evaluations of
 Schubert polynomials for layered permutations, Proc. Amer. Math. Soc.  147 (2019), 1377--1389.

\bibitem{stanley}
R. Stanley,
Some Schubert shenanigans,
arXiv:1704.00851v2, 2017.

\bibitem{Anna}
A. Weigandt,
Schubert polynomials, 132-patterns, and Stanley's conjecture,
Algebr. Combin. 1 (2018), no.4, 415--423.

\bibitem{Zhang-1}
N. Zhang, 
Principal specializations of Schubert polynomials, multi-layered permutations and asymptotics, Adv. Appl. Math. 163 (2025),    102806, 19 pp.
\end{thebibliography}
\end{document}